\documentclass[12pt,centertags,oneside]{amsart}

\usepackage{amscd,amsxtra,calc}
\usepackage{cmmib57}
\usepackage{url}

\usepackage{amscd}
\usepackage{pstricks}
\usepackage{color}
\usepackage[a4paper,width=16.2cm,top=3cm,bottom=3cm]{geometry}

\numberwithin{equation}{section}

\theoremstyle{plain}
\newtheorem{thm}{Theorem}[section]
\newtheorem{theorem}[thm]{Theorem}

\newtheorem{lemma}[thm]{Lemma}

\newtheorem{proposition}[thm]{Proposition}
\theoremstyle{definition}
\newtheorem{remark}[thm]{Remark}

\newtheorem{definition}[thm]{Definition}

\newtheorem{problem}[thm]{Problem}

\numberwithin{equation}{section}

\newcommand{\sO}{{\mathcal O}}

\newcommand{\C}{{\mathbb C}}

\renewcommand{\P}{{\mathbb P}}
\newcommand{\Q}{{\mathbb Q}}
\newcommand{\R}{{\mathbb R}}

\newcommand{\Z}{{\mathbb Z}}

\newcommand{\id}{{\rm id\hspace{.1ex}}}

\newcommand{\Aut}{{\rm Aut\hspace{.1ex}}}
\newcommand{\Inv}{{\rm Inv\hspace{.1ex}}}

\title [Non-finitely generated automorphism group]{A surface with discrete and non-finitely generated\break automorphism group}

\author{Tien-Cuong Dinh}
\address{Department of Mathematics, National University of Singapore, 10, 
Lower Kent Ridge Road, Singapore 119076}
\email{matdtc@nus.edu.sg}
\author{Keiji Oguiso}
\address{Mathematical Sciences, the University of Tokyo, Meguro Komaba 3-8-1, Tokyo, Japan and Korea Institute for Advanced Study, Hoegiro 87, Seoul, 
133-722, Korea}
\email{oguiso@ms.u-tokyo.ac.jp}

\thanks{The first author is supported by NUS Start-Up Grant R-146-000-204-133 and AcRF Tier 1 Grant R-146-000-248-114. 
The second author is supported by JSPS Grant-in-Aid (S) No 25220701, JSPS Grant-in-Aid (S) 15H05738, JSPS Grant-in-Aid (B) 15H03611, and by KIAS Scholar Program. }
 \dedicatory{Dedicated to Professor Shigeyuki Kond\=o on the occasion of his sixtieth birthday}

\begin{document}

\maketitle

\begin{abstract}
We show that there is a smooth complex projective variety, of any dimension greater than or equal to two, whose automorphism group is discrete and not finitely generated. Moreover, this variety admits  infinitely many real forms which are mutually non-isomorphic over $\R$.  Our result is inspired by the work of Lesieutre and answers questions by 
Dolgachev, Esnault and Lesieutre. 
\end{abstract}

\section{Introduction}

In this paper, we work in the category of projective varieties defined over the complex number field $\C$. Our main result is the following theorem which is inspired by the work of Lesieutre and answers questions by  Dolgachev, Esnault and Lesieutre for surfaces, i.e., the cases where $d=2$ in Theorem \ref{thm1}, and their real forms. 

\begin{theorem}\label{thm1}
For each integer $d \ge 2$, there is a smooth projective variety $V$ of dimension $d$ and  of Kodaira dimension $d-2$ such that
\begin{enumerate}
\item the automorphism group ${\rm Aut}\, (V)$ is discrete, i.e., its identity component ${\rm Aut}^{0}\,(V)$ is $\{\id_V\}$; 
\item ${\rm Aut}\, (V)$ is not finitely generated; and
\item $V$ admits infinitely many real forms which are mutually non-isomorphic over $\R$. 
\end{enumerate}
\end{theorem}

Theorem \ref{thm1} is inspired by a remarkable construction, due to Lesieutre, of such a variety in dimension $6$ over any field of characteristic zero  \cite{Le17}, see also Remark \ref{r:field} below. We refer to this reference for history on the finite generation problem of automorphism groups of smooth projective varieties as well as its relation with their real forms. 

For dimension $d=1$, it is well-known that no smooth projective curve satisfies the above properties. 
The existence of surfaces of Kodaira dimension 0 with infinitely many non-isomorphic real forms is unexpected and answers a longstanding open problem. In fact, abelian varieties, surfaces of Kodaira dimension $\geq 1$ and {\it minimal} surfaces of Kodaira dimension 0 only have finitely many non-isomorphic real forms, see Borel-Serre \cite[Cor.6.3]{BS64}, Degtyarev-Itenberg-Kharlamov \cite[Appendix D]{DIK00} and also Benzerga \cite{Ben}. For the reader's convenience, we recall now the notion of real form.

Let $V$ be a variety defined over $\C$. 
Let $V_\R'$ be a variety defined over $\R$ and $V'$ the complex variety  associated to it, i.e.,  
$$V' := V'_{\R} \times_{{\rm Spec}\, \R} {\rm Spec}\, \C\,\, .$$ 
We say that $V_{\R}'$ is {\it a real form} of $V$ if $V'$ is isomorphic to  $V$ over $\C$. Two real forms $V'_\R$ and $V''_\R$ of $V$ are said to be {\it isomorphic} if they are isomorphic over $\R$.

The main new idea to construct varieties of low dimension satisfying Theorem \ref{thm1} is an effective use of a projective K3 surface together with suitable blow-ups. As our projective K3 surface is minimal and defined over $\C$, the canonical representation of its automorphism group is finite, see \cite[Th.14.10]{Ue75}, and blow-ups do not produce new automorphisms. This is an advantage of using a K3 surface, which is not available for the rational surface used in Lesieutre's construction.

In Section \ref{sect2}, we will give a general strategy to construct a variety $V$  from a K3 surface $S$ and a smooth projective variety $M$ of dimension $d-2$ with special properties, see Definition \ref{def23} and Theorem \ref{thm21}. In Section \ref{sect3}, we will present an explicit example of $S$ with the properties required in Theorem \ref{thm21} and get a variety satisfying the properties (1) and (2) in Theorem \ref{thm1}, see Theorem \ref{thm31}.  We will use in these two sections the theory of elliptic surfaces due to Kodaira \cite{Ko63} and classical results from group theory stated in Proposition \ref{prop1}.
In Section \ref{sect4}, we  will specify some parameters in our construction and get the property (3) of Theorem \ref{thm1}.
A criterion for a variety to have infinitely many non-isomorphic real forms, due to Serre and  Lesieutre, is crucial in our approach, see \cite[Chapter 5]{Se02} and  \cite[Lemma 13]{Le17}.

We end this introduction with the following open problems.

\begin{problem} \rm
Is there a smooth projective surface $S$ defined over an algebraically closed field $k$ of positive characteristic such that the automorphism group ${\rm Aut}\, (S)$ is not finitely generated ?
\end{problem}

Our approach uses finiteness of pluricanonical representation and Torelli theorem for K3 surfaces, and can not be directly applied to this question. See also \cite{Le17} for a $6$-dimensional example in positive characteristic. 

\begin{problem} \rm
Are there smooth minimal projective varieties over $\C$  having infinitely many isomorphic classes of real forms  ?
\end{problem}

Such varieties, if exist, should be of dimension  $\ge 3$, by \cite[Appendix D]{DIK00}. It is also unknown if there is a complex rational surface with non-finitely generated ${\rm Aut}\, (S)$ and infinitely many real forms which are mutually non-isomorphic over $\R$.

\medskip\noindent
{\bf Acknowledgements.} We would like to thank Professors  Igor Dolgachev, H\'el\`ene Esnault and John Lesieutre for valuable discussions, questions and suggestions.
The first author would like to thank the University of Tokyo for the hospitality and support during the preparation of this paper.
The second author is grateful to Professors Christopher Hacon, Daniel Huybrechts, Bernd Siebert and Chengyang Xu for their invitation to Oberwolfach, September 2017, at which he learned the work of Professor John Lesieutre. Finally, we would like to thank the referees for their remarks which allow us to improve the presentation of this paper.

\section{General strategy of construction from a special K3 surface}\label{sect2}

In the case of dimension 2, our general strategy is first to blow up a point $P$ of a K3 surface $S$ and get a new surface $S_1$. We then obtain the desired surface $S_2$ by blowing up a finite number of points  in the exceptional curve in $S_1$. Higher dimensional varieties will be obtained as products of $S_2$ and suitable varieties with finite automorphism groups.

By blowing up $S$ and $S_1$, we are able to reduce the automorphism groups of these surfaces. More precisely, with a suitable choices of parameters, $\Aut(S_1)$ is isomorphic to the group of automorphisms of $S$ which fix the point $P$ and $\Aut(S_2)$ is isomorphic to a finite extension of the group of automorphisms of $S$ which fix the point $P$ and also the tangent vectors at $P$. 

Although $\Aut(S)$ is finitely generated, the later group will be infinitely generated for suitable choices of parameters. This allows us to construct varieties satisfying the properties (1) and (2) in Theorem \ref{thm1}. The property (3) in this theorem will be obtained with an extra specification of parameters.

For any projective variety $X$ and any closed algebraic subset $Y$ of $X$, define
$${\rm Aut}\, (X, Y) := \big\{f \in {\rm Aut}\, (X)\, |\, f(Y) = Y \big\}\,\, .$$
This is a subgroup of ${\rm Aut}\, (X)$ and $f|_Y \in {\rm Aut}\, (Y)$ if $f \in {\rm Aut}\, (X, Y)$. 
For simplicity, we denote the group ${\rm Aut}\, (X, \{P\})$ by ${\rm Aut}\, (X, P)$ if $P$ is a closed point of $X$.

From now on, let $S$ be a projective K3 surface, i.e., a smooth projective surface $S$ such that $h^1(S, \sO_S) = 0$ 
and  $\sO_S(K_S) \simeq \sO_S$. We also set $H^0(S, \Omega_S^2) = \C \omega_S$, where $\omega_S$ is a nowhere degenerate holomorphic 2-form. 

Note that ${\rm Aut}\, (S) = {\rm Bir}\, (S)$, i.e., any birational self-map of $S$ is biregular, as $K_S$ is trivial so that no curve is  contracted nor extracted, cf. \cite{Ka08}. In particular, for any birational morphism $p : S' \to S$ from a smooth projective surface $S'$ to $S$, we obtain the following injective group homomorphism 
$${\rm Aut}\, (S') \to {\rm Aut}\, (S)\,\, ;\,\, f \mapsto p \circ f \circ p^{-1}\,\, .$$ 
In what follows, we regard ${\rm Aut}\, (S')$ as a subgroup of ${\rm Aut}\, (S)$ via this group homomorphism. 

Note also that since 
$$H^0(S, T_S) \simeq H^0(S, \Omega_S^1) \simeq \overline{H^1(S, \sO_S)} = \{0\}\, ,$$
$S$ admits no holomorphic vector field and therefore, ${\rm Aut}\, (S)$ is discrete.
So any subgroup of ${\rm Aut}\, (S)$ is discrete as well. 

\smallskip

We now introduce some crucial notions for our construction.

\begin{definition}\label{def21}
A triple $(S, C, P)$ is said to be  {\it special} if 
\begin{enumerate}
\item $S$ is a projective K3 surface, $C \subset S$ is a curve such that $C\simeq \P^1$,  $P \in C$ is a closed point; and
\item ${\rm Aut}\,(S, P) \subset {\rm Aut}\, (S, C)$.
\end{enumerate}
\end{definition}

If $(S, C, P)$ is a special triple as above, then we have the following two natural representations defined by the differentials at $P$
$$r_{S, P} : {\rm Aut}\,(S, P) \to {\rm GL}\, (T_{S, P}) \simeq {\rm GL}(2, \C)\,\, ;\,\, f \mapsto df_P$$
and 
$$r_{C, P} : {\rm Aut}\,(S, P) \to {\rm GL}\, (T_{C, P}) = \C^{\times}\,\, ;\,\, f \mapsto d(f|_C)_P\,\, ,$$
where $T_{S,P}$ is the tangent space of $S$ at $P$ and $T_{C,P}$ is the tangent line of $C$ at $P$. 
For the second representation, we use that $T_{C,P}$ is invariant under the action of $ {\rm Aut}\,(S, P)$ as $(S,C,P)$ is special.
The following  two groups are also important in our study.

\begin{definition}\label{def20} Define
$$G(S, P) := {\rm Ker}\, \big(r_{S, P} : {\rm Aut}\,(S, P) \to {\rm GL}\, (T_{S, P}) \big)\,\, $$ 
and
$$G(S, C, P) := {\rm Ker}\, \big(r_{C, P} : {\rm Aut}\,(S, P) \to {\rm GL}\, (T_{C, P}) \big)\,\, .$$ 
We say that $G(S, C, P)$ is the {\it group associated with the special triple} $(S, C, P)$. 
\end{definition}

\begin{definition}\label{def22}
A special triple $(S, C, P)$ is said to be {\it very special} if 
\begin{enumerate}
\item 
there is a basis $\{ v_1, v_2 \}$ of $T_{S, P}$ with $v_1 \in T_{C, P}$ on which
all the elements of $r_{S, P}({\rm Aut}\, (S, P))$ are diagonal, i.e., 
for every $f \in {\rm Aut}\,(S, P)$,  there are $\alpha_i(f) \in \C$ such that
$$r_{S, P}(f)(v_i) = \alpha_i(f)v_i\,\quad \text{for } i=1, 2\,\, ;$$
and 
\item the group $G(S, C, P)$ associated with $(S, C, P)$ is not finitely generated. 
\end{enumerate}
\end{definition}

In what follows, we will fix a local holomorphic coordinate system $(z_1,z_2)$ centred at the point $P$ such that $v_1=\partial/\partial z_1$ and $v_2=\partial/\partial z_2$.

\begin{remark}\label{rem21}It is known that the automorphism group ${\rm Aut}\, (S)$ of any K3 surface $S$ is finitely generated, see \cite{St85}. However, it is not true in general that any subgroup of ${\rm Aut}\, (S)$ is again finitely generated. In fact, we will construct in  Section \ref{sect3}
a very special triple $(S, C, P)$.
\end{remark}
 
The following classical results from group theory will be used in the proofs of the non-finite generation for groups of automorphisms.
 
\begin{proposition}\label{prop1}
Let $G$ be a group, $H \subset G$ a subgroup of $G$, and $\varphi : G \to Q$ a group homomorphism onto a group $Q$. 
Then the following properties hold.
\begin{enumerate}
\item When the index of $H$ in $G$ is finite,  i.e., $[G : H] < \infty$, the group $H$
 is finitely generated if and only if $G$ is finitely generated.  
\item If $G$ is abelian and finitely generated, then $H$ is also finitely generated regardless of the index $[G:H]$.
\item If $G$ is finitely generated, then $Q$ is always finitely generated. 
\end{enumerate}
\end{proposition}
\proof
For the first assertion, it is obvious that if $H$ is finitely generated then so is $G$; for the converse, 
see e.g. \cite[Page 181, Cor. 1]{Su82}. 
The second assertion is a consequence of the fundamental theorem of finitely generated abelian groups, see e.g. \cite[Chap. 2, Sect. 5]{Su82}. The last assertion is obvious.
\endproof

\begin{lemma}\label{lem21}
Let $(S, C, P)$ be a very special triple. Then $G(S, P)$ is a finite index subgroup of $G(S, C, P)$. In particular, $G(S, P)$ is not finitely generated. 
\end{lemma}
\begin{proof} 
We only need to prove the first assertion because, by Proposition \ref{prop1}(1), 
the second assertion is then a direct consequence.

It is clear that $G(S, P)$ is a subgroup of $G(S, C, P)$. We check that 
$[G(S, C, P) : G(S, P)]$ is finite. By Definitions \ref{def20} and \ref{def22}(1),
the action of $G(S, C, P)$ on $T_{S, P}$ has the following form 
$$df_P(v_1) = v_1\,\, , \,\, df_P(v_2) = \alpha_2(f)v_2\,\, .$$
In particular, $f \in G(S, P)$  if and only if $f \in G(S, C, P)$ and $\alpha_2(f) = 1$. 

Observe that since $H^0(S,\Omega_S^2)=\C\omega_S$, for every $f\in \Aut(S)$, there is a complex number $\alpha(f)$ such that $f^*\omega_S=\alpha(f)\omega_S$.
Recall that the canonical representation 
$$\rho: {\rm Aut}\, (S) \to {\rm GL}\, (H^0(S, \Omega_S^2)) = \C^{\times}\,\, ,\,\, f \mapsto \alpha(f)\quad  \text{with} \quad  f^*\omega_S = \alpha(f)\omega_S$$
is finite, i.e., $|{\rm Im}\, \rho| < \infty$, by \cite[Th.14.10]{Ue75}. Hence the same property holds for the restriction
$\rho|_{G(S, C, P)}$ of $\rho$ to the subgroup $G(S, C, P)$ of $ {\rm Aut}\, (S)$. 

In the local coordinates fixed after Definition \ref{def22}, we write $\omega_S(P)= \lambda dz_1\wedge dz_2$ with $\lambda\in\C^\times$. We have for $f\in G(S,C,P)$
$$f^*\omega_S(P)= \lambda \alpha_2(f) dz_1\wedge dz_2 = \alpha_2(f)\omega_S(P) 
\quad \text{and hence} \quad \alpha_2(f)=\alpha(f).$$ 
It follows that $G(S, P) = {\rm Ker}\, (\rho|_{G(S, C, P)})$. Thus, $[G(S, C, P) : G(S, P)]$ is finite since the image of $\rho$ is finite. This completes the proof of the lemma.
\end{proof}

\begin{definition}\label{def23}
Let $(S, C, P)$ be a very special triple. Let $p_1 : S_1 \to S$ be the blow-up of $S$ at $P$ and $E_P \simeq \P^1$ the exceptional curve. The above local coordinates $(z_1,z_2)$ induce homogeneous coordinates $[z_1:z_2]$ on $E_P$. 
Define two points $P$ and $P'$ in $E_P$ by $P' := [v_1]=[1:0]$ and $P'':=[v_2]=[0:1]$. Choose  $k$ distinct points $Q_1, \ldots, Q_k$ in $E_P\setminus \{P',P''\}$ for some integer $k\geq 1$.
Let $p_2 : S_2 \to S_1$ be the blow-up of $S_1$ at  $Q_1,\ldots ,Q_k$. We will call $S_2$ a {\it core surface} of $(S, C, P)$. \end{definition}

Our main result in this section is Theorem \ref{thm21} below. The existence of a very special triple will be obtained later  in Section \ref{sect3}. The two results together  give us a variety $V$  satisfying the properties (1) and (2) in  Theorem \ref{thm1} with $V:=S_2$ for $d=2$ and $V:=S_2\times M$ for $d=m+2\geq 3$. In order to get the property (3) in that theorem, we need to specify some parameters in the construction. This will be done in Section \ref{sect4}.

\begin{theorem}\label{thm21}
Let $(S, C, P)$ be a very special triple and $S_2$ a core surface of $(S, C, P)$ as above. Then
\begin{enumerate}
\item ${\rm Aut}\, (S_2)$ is discrete and not finitely generated;
\item for each integer $m \ge 1$, there is a smooth projective variety $M$ of dimension $m$ and general type such that ${\rm Aut}\, (S_2 \times M)$ is discrete and not finitely generated. 
\end{enumerate}
\end{theorem}
\begin{proof}
(1)  We have seen that ${\rm Aut}\, (S_1)$ and ${\rm Aut}\, (S_2)$ can be regarded as  subgroups of ${\rm Aut}\, (S)$ 
via $p_1$ and  $p_1 \circ p_2$. So these groups are discrete since 
${\rm Aut}\, (S)$ is discrete. It remains to check that ${\rm Aut}\, (S_2)$ is not finitely generated. We use the notation from Definitions \ref{def21}, \ref{def22}, \ref{def23} and note that 
$$G(S, P) \subset {\rm Aut}\,(S_2) \subset {\rm Aut}\, (S)$$
as $f|_{T_{S, P}} = \id_{T_{S, P}}$ for $f \in G(S, P)$. 
By Lemma \ref{lem21}, the group $G(S,P)$ is not finitely generated.  
So by Proposition \ref{prop1}(1), it is enough to prove that  $[{\rm Aut}\, (S_2) : G(S, P)]$ is finite.

As $K_S$ is trivial, by the canonical bundle formula, we have 
$$K_{S_2} = E_P' + 2E_{Q_1} + \cdots + 2E_{Q_k}\, ,$$
where $E_P'$ is the proper transform of $E_P$ and $E_{Q_i}$ is the exceptional divisor of the second blow-up $p_2$, lying over $Q_i$ for  $i=1 ,\ldots, k$. Note  that 
$$|K_{S_2}| = \{E_P' + 2E_{Q_1} + \cdots + 2E_{Q_k} \}\,\, .$$ 
As ${\rm Aut}\, (S_2)$ acts on $|K_{S_2}|$, we deduce that for every $f\in  {\rm Aut}\, (S_2)$
$$f(E_P') = E_{P}'\quad \text{and} \quad f(\{E_{Q_1}, \ldots, E_{Q_{k}}\}) =  \{E_{Q_1}, \ldots, E_{Q_{k}}\} \,\, .$$
Therefore, via $p_2$ and $p_1$, we can write 
$${\rm Aut}\, (S_2) = {\rm Aut}\, (S_1, E_P) \cap {\rm Aut}\,(S_1, \{Q_1, \ldots, Q_k\}) \subset {\rm Aut}(S, P)\,\, .$$

Let $f \in {\rm Aut}\,(S_2) \subset {\rm Aut}\, (S, P)$. 
By Definition \ref{def22}(1), $f$ fixes the points $P'$ and $P''$.  So $f|_{E_P}$ fixes two points $P',P''$ and the set $\{Q_1, \ldots, Q_k\}$. 
Consider the subgroup 
$$G_0 := \big\{f \in {\rm Aut}\,(S_2) \subset {\rm Aut}\, (S, P)\,|\, f(Q_i) = Q_i \text{ for } i=1, \ldots, k\, \big\}\,\,$$
of ${\rm Aut}\,(S_2)$. It is clear that $[{\rm Aut}\,(S_2) : G_0] \le k!$. 

Let $f \in G_0$. Its restriction to $E_P$ fixes $k+2$ points. Therefore,
$f|_{E_P} = \id_{E_P}$, or equivalently, $df_P = \alpha(f) \id_{T_{S, P}}$ for some number $\alpha(f)\in \C^\times$. Recall that $G(S, P) \subset G_0$. So it follows from the last identity that  $f \in G(S, P)$ if and only if $\alpha(f) = 1$.
We also have
$$f^*\omega_S(P) = \alpha(f)^2\omega_S(P)\,\, .$$

Now, the group 
$$\big\{\alpha(f)^2\in \C^\times\, |\, f \in G_0 \big\}$$ 
is finite by the finiteness of the canonical representation, see again \cite[Th.14.10]{Ue75}. Thus, the group
$$I :=\big\{\alpha(f)\in \C^\times\, |\, f \in G_0 \big\}$$
is finite as well.  We conclude that  
$$[{\rm Aut}\, (S_2) : G(S, P)] = [{\rm Aut}\, (S_2) : G_0] \cdot [G_0 : G(S, P)] \le k! \cdot |I|\,\, .$$
Hence $[{\rm Aut}\, (S_2) : G(S, P)]$ is finite.
This ends the proof of the assertion (1).

\medskip

(2)  Let $A$ be an abelian variety of dimension $m+1$ and $M$ a smooth very ample divisor of $A$. Then $M$ is an $m$-dimensional smooth variety with ample canonical bundle by the adjunction formula. So it is of general type. It is obvious that such  $A$ and $M$ exist. 

Observe that 
$$H^0(S_2, \Omega_{S_2}^1) \simeq H^0(S, \Omega_S^1) = 0\,\, .$$ 
So any morphism from $S_2$ to an abelian variety is constant. In particular, any morphism from $S_2$ to $M$ is constant.
We will show later in  Lemma \ref{lem25} below that 
$${\rm Aut}(S_2 \times M) = {\rm Aut}\, (S_2) \times {\rm Aut}\, (M)\,\, .$$
As $M$ is of general type, ${\rm Aut}\, (M)$ is a finite group by the finiteness of pluri-canonical representation \cite[Th.14.10]{Ue75}. In particular, ${\rm Aut}\, (S_2)$ is isomorphic to a finite index subgroup of ${\rm Aut}\, (S_2\times M)$. Recall that  ${\rm Aut}\, (S_2)$ is not finitely generated by the assertion (1). 
We deduce from Proposition \ref{prop1}(1)
that ${\rm Aut}\, (S_2 \times M)$ is not finitely generated. This completes the proof of the theorem.
\end{proof}

We used above the following lemma. 

\begin{lemma}\label{lem25}
Let $S$ be a projective K3 surface and  $\nu : S' \to S$ a birational morphism from a smooth surface $S'$ to $S$. Let $M$ be a smooth projective variety. Assume that there is no non-constant morphism from $S'$ to $M$. 
Then 
$${\rm Aut}\, (S' \times M) = {\rm Aut}\, (S') \times {\rm Aut}\, (M)\,\, .$$
\end{lemma}
\begin{proof}
Let $f \in {\rm Aut}\, (S' \times M)$. It suffices to show that $f$ is of the form $f = (f_{S'}, f_M)$ with $f_{S'} \in {\rm Aut}\, (S')$ and $f \in {\rm Aut}\, (M)$.  
We deduce from the assumption that the natural fibration $\pi : S' \times M \to M$ is preserved under ${\rm Aut}\, (S' \times M)$. Therefore, $f$ is of the form 
$$f(x, y) = (f_y(x), f_M(y))\,\,  \text{with } x \in S' \text{ and } y \in M\,\, ,$$
where $f_M \in {\rm Aut}\, (M)$ and 
$f_y : S'  \to S' $
is an isomorphism. Hence we have a continuous map
$$F : M \to {\rm Aut}\, (S')\,\, ,\,\, y \mapsto f_y\,\, .$$
As $M$ is connected and ${\rm Aut}\, (S') \subset {\rm Aut}\, (S)$ is discrete, $F$ is a constant map. Thus, we can write  $f = (f_{S'}, f_M)$ for some $f_{S'} \in {\rm Aut}\, (S')$ and $f_M \in {\rm Aut}\, (M)$, as claimed.  
\end{proof}

We will end this section with the following lemma that will be used later in Section \ref{sect4}.

\begin{lemma} \label{l:S2-EP}
The image of $\Aut(S_2)$ by $r_{C,P}$ is a finite subset of $\C^\times$. 
\end{lemma}
\proof
Consider a map $f$ in $\Aut(S_2)$. As in Lemma \ref{lem21}, there are two numbers $\alpha_1(f)$ and $\alpha_2(f)$ in $\C^\times$ such that 
$$df_P(v_1)=\alpha_1(f)(v_1) \quad \text{and} \quad df_P(v_2)=\alpha_2(f)(v_2) \, .$$
These identities also imply that
$$f^*\omega_S= \alpha_1(f)\alpha_2(f) \omega_S \, .$$
Moreover, also as in Lemma  \ref{lem21}, the number $\alpha_1(f)\alpha_2(f)$ belongs to a finite subset of $\C^\times$. 

Now, consider the affine coordinate $z:=z_2/z_1$ on $E_P$. The restriction of $f$ to $E_P$ is then given by 
$$z\mapsto \alpha_1(f)^{-1}\alpha_2(f) z \, .$$
Since it preserves the set $\{Q_1,\ldots,Q_k\}$, the number $\alpha_1(f)^{-1}\alpha_2(f)$ belongs to a finite set which depends on 
$\{Q_1,\ldots,Q_k\}$. We conclude that $\alpha_1(f)$ and $\alpha_2(f)$ belong to a finite subset of $\C^\times$. The lemma follows.
\endproof

\section{Existence of a very special triple}\label{sect3} 

In this section, we show that there is a very special triple $(S, C, P)$ as in Definition \ref{def22}. 
From now on, consider the Kummer surface
$$S = {\rm Km}\, (E \times F)$$ 
associated to the product of two {\it mutually non-isogenous} elliptic curves $E$ and $F$. Then $S$ is a projective K3 surface of Picard number $18$, see e.g. \cite{Og89}. 
We specify the affine Weierstrass equation of $E$ with origin at $\infty$ as 
$$y^2 = x(x-1)(x-2)\quad \text{with} \quad (x,y)\in\C^2\,\, .$$ 
Note that $E/\langle -1_E \rangle = \P^1$ and 
the associated quotient map $E \to \P^1$ is given in the above affine coordinates by $(x,y)\mapsto x$. Moreover,
the points $0$, $1$, $2$ and $\infty$ of $\P^1$ are the branch points of this map.

Let $\{a_i\}_{i=1}^{4}$ and $\{b_i\}_{i=1}^{4}$ be the $2$-torsion subgroups of $F$ and $E$ respectively. Then $S$ contains the  following 24 smooth rational curves. There are $8$ smooth rational curves $E_i$, $F_i$, with $1 \le i \le 4$, arising from $8$ elliptic curves $E \times \{a_i\}$, $\{b_i\} \times F$ on $E \times F$, and $16$ exceptional curves $C_{ij}$, with $1\le i,j \le 4$, over the $16$ singular points of type $A_1$ on the quotient surface $E \times F/\langle -1_{E \times F}\rangle$. See Figure \ref{fig1} for the configuration of these $24$ smooth rational curves on $S$. 

In what follows, we use the notation in Figure \ref{fig1} and additionally set

\begin{figure}
\unitlength 0.1in
\begin{picture}(25.000000,24.000000)(-1.000000,-23.500000)
\put(4.500000, -22.000000){\makebox(0,0)[rb]{$F_1$}}%
\put(9.500000, -22.000000){\makebox(0,0)[rb]{$F_2$}}%
\put(14.500000, -22.000000){\makebox(0,0)[rb]{$F_3$}}%
\put(19.500000, -22.000000){\makebox(0,0)[rb]{$F_4$}}%
\put(0.250000, -18.500000){\makebox(0,0)[lb]{$E_1$}}%
\put(0.250000, -13.500000){\makebox(0,0)[lb]{$E_2$}}%
\put(0.250000, -8.500000){\makebox(0,0)[lb]{$E_3$}}%
\put(0.250000, -3.500000){\makebox(0,0)[lb]{$E_4$}}%
\put(6.000000, -16.000000){\makebox(0,0)[lt]{$C_{11}$}}%
\put(6.000000, -11.000000){\makebox(0,0)[lt]{$C_{12}$}}%
\put(6.000000, -6.000000){\makebox(0,0)[lt]{$C_{13}$}}%
\put(6.000000, -1.000000){\makebox(0,0)[lt]{$C_{14}$}}%
\put(11.000000, -16.000000){\makebox(0,0)[lt]{$C_{21}$}}%
\put(11.000000, -11.000000){\makebox(0,0)[lt]{$C_{22}$}}%
\put(11.000000, -6.000000){\makebox(0,0)[lt]{$C_{23}$}}%
\put(11.000000, -1.000000){\makebox(0,0)[lt]{$C_{24}$}}%
\put(16.000000, -16.000000){\makebox(0,0)[lt]{$C_{31}$}}%
\put(16.000000, -11.000000){\makebox(0,0)[lt]{$C_{32}$}}%
\put(16.000000, -6.000000){\makebox(0,0)[lt]{$C_{33}$}}%
\put(16.000000, -1.000000){\makebox(0,0)[lt]{$C_{34}$}}%
\put(21.000000, -16.000000){\makebox(0,0)[lt]{$C_{41}$}}%
\put(21.000000, -11.000000){\makebox(0,0)[lt]{$C_{42}$}}%
\put(21.000000, -6.000000){\makebox(0,0)[lt]{$C_{43}$}}%
\put(21.000000, -1.000000){\makebox(0,0)[lt]{$C_{44}$}}%
\special{pa 500 2200}%
\special{pa 500 0}%
\special{fp}%
\special{pa 1000 2200}%
\special{pa 1000 0}%
\special{fp}%
\special{pa 1500 2200}%
\special{pa 1500 0}%
\special{fp}%
\special{pa 2000 2200}%
\special{pa 2000 0}%
\special{fp}%
\special{pa 0 1900}%
\special{pa 450 1900}%
\special{fp}%
\special{pa 550 1900}%
\special{pa 950 1900}%
\special{fp}%
\special{pa 1050 1900}%
\special{pa 1450 1900}%
\special{fp}%
\special{pa 1550 1900}%
\special{pa 1950 1900}%
\special{fp}%
\special{pa 0 1400}%
\special{pa 450 1400}%
\special{fp}%
\special{pa 550 1400}%
\special{pa 950 1400}%
\special{fp}%
\special{pa 1050 1400}%
\special{pa 1450 1400}%
\special{fp}%
\special{pa 1550 1400}%
\special{pa 1950 1400}%
\special{fp}%
\special{pa 0 900}%
\special{pa 450 900}%
\special{fp}%
\special{pa 550 900}%
\special{pa 950 900}%
\special{fp}%
\special{pa 1050 900}%
\special{pa 1450 900}%
\special{fp}%
\special{pa 1550 900}%
\special{pa 1950 900}%
\special{fp}%
\special{pa 0 400}%
\special{pa 450 400}%
\special{fp}%
\special{pa 550 400}%
\special{pa 950 400}%
\special{fp}%
\special{pa 1050 400}%
\special{pa 1450 400}%
\special{fp}%
\special{pa 1550 400}%
\special{pa 1950 400}%
\special{fp}%
\special{pa 200 2000}%
\special{pa 600 1600}%
\special{fp}%
\special{pa 200 1500}%
\special{pa 600 1100}%
\special{fp}%
\special{pa 200 1000}%
\special{pa 600 600}%
\special{fp}%
\special{pa 200 500}%
\special{pa 600 100}%
\special{fp}%
\special{pa 700 2000}%
\special{pa 1100 1600}%
\special{fp}%
\special{pa 700 1500}%
\special{pa 1100 1100}%
\special{fp}%
\special{pa 700 1000}%
\special{pa 1100 600}%
\special{fp}%
\special{pa 700 500}%
\special{pa 1100 100}%
\special{fp}%
\special{pa 1200 2000}%
\special{pa 1600 1600}%
\special{fp}%
\special{pa 1200 1500}%
\special{pa 1600 1100}%
\special{fp}%
\special{pa 1200 1000}%
\special{pa 1600 600}%
\special{fp}%
\special{pa 1200 500}%
\special{pa 1600 100}%
\special{fp}%
\special{pa 1700 2000}%
\special{pa 2100 1600}%
\special{fp}%
\special{pa 1700 1500}%
\special{pa 2100 1100}%
\special{fp}%
\special{pa 1700 1000}%
\special{pa 2100 600}%
\special{fp}%
\special{pa 1700 500}%
\special{pa 2100 100}%
\special{fp}%
\end{picture}%
 \caption{Curves $E_i$, $F_j$ and $C_{ij}$}
 \label{fig1}
\end{figure}
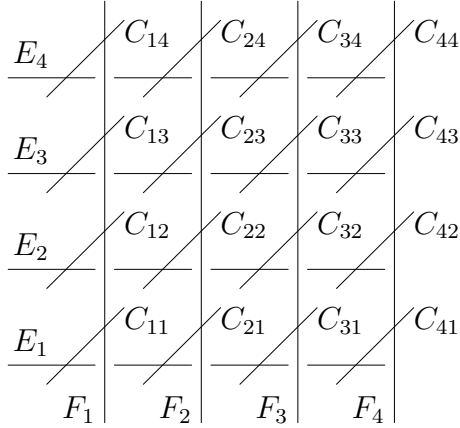

$$C := E_1 \subset S\,\, .$$
Since $C =  E/\langle -1_E \rangle$, we can use $x$ as an affine coordinate on $C$. In this coordinate,
the set of the four points 
$$\big\{ C \cap C_{11}, C \cap C_{21}, C \cap C_{31}, C \cap C_{41} \big\}$$ 
coincides with the set of branch points
$$\{0, 1, 2, \infty\} \, .$$

\begin{definition}\label{def31}
Using a renumeration of the 24 curves, if necessary, we can assume that in the affine coordinate $x$ 
$$C \cap C_{11} = \infty\,\, , \,\, C \cap C_{21} = 0\,\, ,\,\, C \cap C_{31} = 1 \,\, , \,\, C \cap C_{41} = 2\,\, .$$
We set $P$ equal to the first point, that is, 
$$P:=C \cap C_{11} = \infty\,\, .$$
\end{definition}

The main result of this section is the following result. Together with Theorem \ref{thm21}, it gives us a variety $V$ satisfying the properties (1) and (2) in Theorem \ref{thm1}. Later in Section \ref{sect4}, some more specific choices of parameters will be made  in order to get the property (3) in Theorem \ref{thm1}. 

\begin{theorem}\label{thm31} The triple $(S, C, P)$ is a very special triple. 
\end{theorem}

The proof of this theorem occurs the rest of this section. It is partly based on \cite{Og89} and some arguments employed by \cite{Le17} and \cite{Og17}. 
We will check all requirements for a very special triple in Definition \ref{def22}.

By definition, $S$ is a smooth projective K3 surface, $C \subset S$ is a smooth rational curve and $P \in C$. 
Let $\theta$ be the automorphism of $S$ induced by the automorphism $(1_E, -1_F)$ of $E \times F$. Set 
$$B := \cup_{i=1}^{4} E_i \cup \cup_{j=1}^{4} F_j\,\, .$$

\begin{lemma}\label{lem31} 
The following properties hold.
\begin{enumerate}
\item $\theta^* = \id$ on ${\rm Pic}\, (S)$ and $\theta^* \omega_S = -\omega_S$.\item $f \circ \theta = \theta \circ f$ for all $f \in {\rm Aut}\, (S)$. 
\item The fixed locus of $\theta$ is exactly the set $B$. 
\item ${\rm Aut}(S) = {\rm Aut}\, (S, B)$. 
\end{enumerate}
\end{lemma}
\begin{proof} 
This lemma was obtained in \cite[Lemmas (1.3), (1.4)]{Og89}.   We briefly recall here the proof since some ideas will be used later.

Observe that ${\rm Pic}\, (S)$ is a torsion free group, ${\rm Pic}\, (S) \otimes \Q$ is generated by the above $24$ rational curves and $\theta$ preserves these $24$ curves. It follows immediately that  $\theta^* = \id$ on ${\rm Pic}\, (S)$. By the definition of $\theta$, we also obtain $\theta^*\omega_S = -\omega_S$. This shows (1). 

By (1), we have $\theta^* \circ f^* = f^* \circ \theta^*$ on ${\rm Pic}\, (S)$ and also on $H^0(S, \Omega_S^2)$ because $h^0(S, \Omega_S^2) = 1$. Recall that ${\rm Pic}\, (S) \simeq {\rm NS}\, (S) \subset H^2(S, \Z)$ and 
$$\big( {\rm NS}\, (S) \otimes \Q  \big) \oplus  \big( {\rm NS}\, (S)^{\perp}_{H^2(S, \Z)} \otimes \Q \big) = H^2(S, \Q)\,\, .$$ 
Moreover, by Lefschetz' $(1, 1)$-theorem, ${\rm NS}\, (S)^{\perp}_{H^2(S, \Z)}$ is the minimal primitive sublattice of $H^2(S, \Z)$ such that 
$$\C[\omega_S] \in {\rm NS}\, (S)^{\perp}_{H^2(S, \Z)} \otimes \C\,\, .$$ 
It follows that $\theta^* = -\id$ on ${\rm NS}\, (S)^{\perp}_{H^2(S, \Z)}$. Therefore, $\theta^* \circ f^* = f^* \circ \theta^*$ on $H^2(S, \Q)$, and also 
on $H^2(S, \Z)$ as this group is torsion free. Thus, $f \circ \theta = \theta \circ f$ by the global Torelli theorem for K3 surfaces, see e.g. \cite[Chapter VIII, Sect. 11]{BHPV04}. This proves (2). 

The assertion (3) is geometrically immediate from the shape of $\theta$. The assertion (4) follows from (2) and (3). 
\end{proof}

\begin{lemma}\label{lem32} 
We have ${\rm Aut}\, (S, P) \subset {\rm Aut}\, (S, C)$, or equivalently, $f(C) = C$ for every
$f \in {\rm Aut}\, (S, P)$.  In particular, $(S, C, P)$ is a special triple as in Definition \ref{def21}. 
\end{lemma}
\begin{proof} By Lemma \ref{lem31}, we have ${\rm Aut}(S) = {\rm Aut}(S, B)$. 
As $C$ is the unique curve containing $P$ among the $8$ curves in $B$, the result follows.  
\end{proof}

\begin{lemma}\label{lem33} Let $R \subset S$ be a smooth rational curve such that $R \not\subset B$. Then the following properties hold.
\begin{enumerate}
\item $\theta(R) = R$ and $\theta|_R \in {\rm Aut}\, (R)$ is of order $2$ with exactly two fixed points. Moreover, $d(\theta|_R)_Q = -1$ at each fixed point $Q$ of $\theta|_R$. 
\item Assume furthermore that $P \in R$. Then ${\rm mult}_P(R, C) = 1$ and ${\rm mult}_P(f(R), R) \ge 2$ for all $f \in {\rm Aut}\, (S, P)$. \end{enumerate}
Here, the notation ${\rm mult}_P$ stands for the local multiplicity of intersection at the point $P$ with the convention that ${\rm mult}_P(f(R), R) = \infty$ when $f(R) = R$. 
\end{lemma}
\begin{proof} 
By the adjunction formula for K3 surfaces, any smooth rational curve in $S$ has self-intersection $-2$. 
We then deduce that $|R| = \{R\}$.
As $(\theta^{-1})^* = \id$ on ${\rm Pic}\, (S)$, it follows that $(\theta^{-1})^*R \in |R|$, from which the identity $\theta(R) = R$ follows. As $\theta$ is of order $2$ and $R$ is not contained in the fixed locus $B$ of $\theta$, the order of $\theta|_R$ is exactly $2$. As $R \simeq \P^1$, the involution $\theta|_R$ has exactly two fixed points and $d(\theta|_R)_Q = -1$ at any fixed point $Q$. This proves (1). 

Observe that $d(\theta|_C)_P = 1$  by $\theta|_C = \id_C$ and $d(\theta|_R)_P = -1$ by (1). It follows that ${\rm mult}_P(R, C) = 1$. Let $f$ be an automorphism in ${\rm Aut}\, (S, P)$. We have $f(P)=P$ and hence $P \in f(R)$ since $P\in R$.
By Lemma \ref{lem31}(4), we have $f(B) = B$.  Using that $R\not\subset B$, we obtain that $f(R) \not\subset B$ because $f$ is an automorphism. Applying (1) to $f(R)$ instead of $R$ gives $d(\theta|_{f(R)})_P = -1$. Finally, as $d(\theta|_R)_P = d(\theta|_{f(R)})_P =-1$ and $d(\theta|_C)_P = 1$, the two curves $f(R)$ and $R$ cannot intersect transversally at $P$. It follows that ${\rm mult}_P(f(R), R)$ is at least 2.  
\end{proof}

Recall that the representation $r_{S, P}$ is defined by
$$r_{S, P} : {\rm Aut}\, (S, P) \to {\rm GL}(T_{S, P})\,\, ;\,\, f \mapsto df_P\,\, .$$

\begin{lemma}\label{lem34} 
${\rm Im}\, (r_{S, P})$ is simultaneously diagonalizable in the sense of Definition \ref{def22}(1). 
\end{lemma}
\begin{proof} Choose $v_1 \in T_{C, P} \setminus \{0\}$ and $v_2 \in T_{R, P}\setminus \{0\}$ with $R:=C_{11}$. Then $\{v_1, v_2 \}$ is a basis of $T_{S, P}$ and by Lemma \ref{lem33}(2), $df_P$ preserves $\C v_1$ and $\C v_2$.  Thus, the result follows.  
\end{proof}

In order to complete the proof of Theorem \ref{thm31}, it remains to check the requirement in  Definition \ref{def22}(2).
We have the following proposition. 

\begin{proposition}\label{prop31}
The group $G(S, C, P)$ associated to the triple $(S, C, P)$, introduced in Definition \ref{def20}, is not finitely generated. 
\end{proposition}

For the proof, we will need the following two results on elliptic surfaces.

\begin{proposition}\label{prop33}
Let $V$ be a K3 surface and $D$ an effective divisor on $V$ such that $D$ is nef and $(D^2) = 0$. Assume that there is a smooth rational curve $O$ on $V$ such that $(D.O) = 1$. Then there is an elliptic fibration $\varphi_D : V \to \P^1$ such that $D$ is a fiber of $\varphi_D$ and $O$ is a global section of $\varphi_D$. Moreover, this elliptic fibration is necessarily minimal, i.e., no fiber contains a smooth rational curve of self-intersection $-1$.  
\end{proposition}
\noindent
{\it Sketch of the proof.}
This is well-known to experts. We just sketch an elementary proof for the reader's convenience. 

Since  $D$ is non-zero, nef, effective and $(D^2) = 0$, we obtain from Riemann-Roch theorem and Serre duality that
$h^0(V, \sO_V(D)) \ge 2$. 
Let $D'$ be the movable part of $D$. Then $D'$ is nef. As $D$ is nef, we deduce that $(D.D') = 0$ from 
$$0 = (D^2) = (D.D') + (D.(D-D')) \ge 0 + 0 = 0\,\, . $$ 
Then, using the Hodge index theorem for nef divisors $D$ and $D'$, we further deduce that  the classes $(D)$ and $(D')$ are co-linear  in ${\rm NS}\, (V)$. It is not difficult to obtain that $D =D'$ using $(D.O)=1$; and hence 
 $|D|$ has no fixed component. Thus, $|D|$ is free again by $(D^2) = 0$. 
 
 The morphism $\varphi_D$ given by $|D|$ is a fibration having $D$ as a fiber and its smooth fibers are elliptic curves by the adjunction formula. It follows that $\varphi_D$ is an elliptic fibration. It is also relatively minimal because $K_V$ is trivial.  
 Finally, since $(D.O) = 1$ and $|D|$ is free, $O$ cannot be a component of $D$ and hence it is a global section. As $O\simeq \P^1$, $\varphi_D$ defines a fibration over $\P^1$. 
\hfill $\square$

\medskip

Consider a relatively minimal elliptic surface $\varphi : V \to \P^1$ with a global zero section $O$. Denote by $V_t$ the fiber over $t \in \P^1$ and  by $O_t$ the unique point in the intersection $V_t\cap O$.
When  $V_t$ is smooth, it is an elliptic curve and the point $O_t$ is regarded as the origin of  $V_t$. Similarly, if $U$ is an arbitrary global  section of $\varphi$, denote by $U_t$ the unique point in the intersection $V_t\cap U$.
Finally, let $V_{t, 0}$ denote the unique irreducible component of $V_t$ containing $O_t$. We have $V_{t,0}=V_t$ when $V_t$ is smooth.
The following proposition  is a special case of  a more general result due to Kodaira, see \cite[Th.9.1. page 604]{Ko63}.

\begin{proposition}\label{prop32}
The following properties hold.
\begin{enumerate}
\item The set of global sections of $\varphi$ forms an abelian group ${\rm MW}(\varphi)$, called the Mordell-Weil group, 
 that can be identified to a subgroup of  ${\rm Aut}\, (V)$. The action of an element $U \in {\rm MW}(\varphi)$ on $V$ is  induced by the addition 
$$z \mapsto z + U_t$$
on smooth fibers $V_t$. 
\item Let $V_s$ be a singular fiber of additive type, i.e., a singular fiber such that 
$$V_{s, 0} \setminus {\rm Sing}\, (V_s) \simeq \C\,\, .$$
Let $U \in {\rm MW}(\varphi)$ be a section which meets $V_{s,0}$. 
Then $V_{s, 0} \cap {\rm Sing}\, (V_s)$ consists of a single point and the action of $U$ on $V_{s, 0} \setminus {\rm Sing}\, (V_s)$ is given by the addition
$$z \mapsto z + U_s$$ 
with respect to any affine coordinate $z$ of $V_{s, 0} \setminus {\rm Sing}\, (V_s)$ in which
$$O_s = 0\quad \text{and}\quad V_{s, 0} \cap {\rm Sing}\, (V_s) = \infty\,\, .$$ 
Such a coordinate $z$ is unique up to non-zero constant multiple  (as $V_{s, 0}$ is rational). 
\item Let $V_s$ be a  reducible singular fiber of multiplicative type, i.e., a reducible singular fiber such that 
$$V_{s, 0} \setminus {\rm Sing}\, (V_s) \simeq \C^{\times}\,\, .$$  
Let $U \in {\rm MW}(\varphi)$ be a section which meets $V_{s, 0}$. 
Then $V_{s, 0} \cap {\rm Sing}\, (V_s)$ consists of two points and the action of $U$ on $V_{s, 0} \setminus {\rm Sing}\, (V_s)$ is given by the multiplication
$$z \mapsto z \times U_s$$ 
with respect to any affine coordinate $z$ of $V_{s, 0} \setminus {\rm Sing}\, (V_s)$ in which 
$$O_s = 1\quad \text{and} \quad  V_{s, 0} \cap {\rm Sing}\, (V_s) = \{0, \infty\}\,\, .$$ 
Such a coordinate $z$ is unique up to the inversion $z \mapsto z^{-1}$ (as again $V_{s, 0}$ is rational). 
\end{enumerate}
\end{proposition}

\noindent
{\it Proof of Proposition \ref{prop31}}. 
Consider the group representation
$$\tau : G(S, C, P) \to {\rm Aut}(C, P)\,\, ;\,\, f \mapsto f|_C\,\, $$
and set 
$$\Gamma := \tau(G(S, C, P))\,\, .$$ 
Let $f \in G(S, C, P)$. 
Recall that $C = E_1 \simeq \P^1$ and we already fixed the affine coordinate $x$ of $C$ at the beginning of this section  (Definition \ref{def31}).
By definition of $G(S, C, P)$, 
the automorphism $f|_C$ is of the form 
$$f(x) = x + a_f\quad \text{with} \quad a_f \in \C.$$
In particular, $\Gamma$ is abelian as it is isomorphic to a subgroup of $\C$. 
Therefore, by Proposition \ref{prop1}(2)-(3), in order to get that
$G(S, C, P)$ is not finitely generated, it is enough to check that $\Gamma$ admits a non-finitely generated subgroup. 
 
Consider the following two divisors of Kodaira's type on $S$
$$D_1 := C + C_{11} + F_1 + C_{12} + E_2 + C_{22} + F_2 + C_{21}$$
and
$$D_2 := C + 2C_{11} + E_2 + 2C_{12} + E_3 + 2C_{13} + 3F_1\,\, .$$
Using Figure 1, we see that in $D_1$, the curve $C$ intersects the two components $C_{11}$ and $C_{21}$; in $D_2$, the curve $C$ intersects only the component $C_{11}$. It is also easy to check that $(D_1^2)=(D_2^2) = 0$ and $D_1$, $D_2$ are nef, because their intersection number with each irreducible component is $0$. 

Observe moreover that 
$$(D_1.C_{31}) = (D_1.C_{41}) = 1$$
and
$$(D_2.C_{21}) = (D_2.C_{31}) = 1\,\, .$$
Thus, thanks to Proposition \ref{prop33}, we obtain two elliptic fibrations 
$$\varphi_{D_1} : S \to \P^1$$
with $D_1$ as a singular fiber and two global sections $C_{31}$, $C_{41}$ meeting $C$, and 
$$\varphi_{D_2} : S \to \P^1$$
with $D_2$ as a singular fiber and two global sections $C_{21}$, $C_{31}$ meeting $C$. 

We now apply Proposition \ref{prop32}. 
Choose $C_{31}$ as the zero section of $\varphi_{D_1}$ and $C_{21}$ as the zero section of $\varphi_{D_2}$. We see that $D_1$ is a singular fiber of multiplicative type $I_8$ for $\varphi_{D_1}$
and $D_2$ is a singular fiber of additive type $IV^*$ for $\varphi_{D_2}$. The following automorphisms play crucial roles in the proof of Proposition \ref{prop31} and also in Section \ref{sect4}.

\begin{definition}\label{def32}
Let $f_1$ and $f_2$ denote the automorphisms of $S$ given respectively by 
$C_{41} \in {\rm MW}(\varphi_{D_1})$ and $C_{31} \in {\rm MW}(\varphi_{D_2})$. 
\end{definition} 

Note that the affine coordinate $x$ and the coordinate values 
$$C \cap C_{11} = \infty\,\, ,\,\,  C \cap C_{21} = 0\,\, ,\,\, C \cap C_{31} = 1\,\, ,\,\, C \cap C_{41} = 2$$
in Definition \ref{def31} are consistent with the coordinate $z$ in Proposition \ref{prop32} for both $\varphi_{D_1}$ and $\varphi_{D_2}$. 
Both $f_1$ and $f_2$ preserve $C$ and the induced actions on $C$ are given by 
$$f_1|_C(x) = 2x \quad \text{and} \quad f_2|_C(x) = x+1\,\, .$$
In particular, we see that  $f_1, f_2 \in {\rm Aut}\, (S, P)$. 

Observe that 
$$(f_1|_C)^{-n} \circ (f_2|_C) \circ (f_1|_C)^{n}(x) =  x + \frac{1}{2^n}\,\, ,$$
for any positive integer $n$. 
Thus, $f_1^{-n} \circ f_2 \circ f_1^{n} \in G(S, C, P)$ and therefore, 
$$(f_1|_C)^{-n} \circ (f_2|_C) \circ (f_1|_C)^{n} \in \Gamma\,\, .$$ 
Hence the group 
$$\Gamma':=\big\langle (f_1|_C)^{-n} \circ (f_2|_C) \circ (f_1|_C)^{n}\, |\, n \in \Z_{>0} \big\rangle$$ 
is a subgroup of $\Gamma$ and it is isomorphic to the subgroup
$$\Gamma'' := \big\langle \frac{1}{2^n}\, |\, n \in\Z_{>0} \big\rangle$$
of the additive group $\Q$. The group $\Gamma''$ is not finitely generated. Indeed, otherwise, the denominators of its elements must be uniformly bounded. Thus, $\Gamma'$ is not finitely generated as well. This ends the proof of Proposition \ref{prop31}. 
\hfill $\square$

\medskip\noindent
{\it Proof of Theorem \ref{thm31}.} 
This is a consequence of  Lemmas \ref{lem32}, \ref{lem34} and Proposition \ref{prop31}. 
\hfill $\square$

\section{A variety with infinitely many real forms}\label{sect4}

In this section, we will specify some parameters in our construction of a very special triple in order to get a variety $V$ satisfying the property (3) in Theorem \ref{thm1}, see the properties (S1)-(S5) below.
We continue to use the same notation as in the previous sections and need to study some involutions of $S$ as in the work by Lesieutre \cite{Le17}. 

Consider an involution $f$ in ${\rm Aut} \, (S)$. Since $H^0(S,\Omega_S^2)=\C\omega_S$, the form $f^*(\omega_S)$ is equal to a constant times $\omega_S$. This constant should be 1 or $-1$ because $f$ is an involution.

\begin{lemma} \label{l:inv}
Let $f$ be any involution of $S$ such that $f^*\omega_S=-\omega_S$. Then the differential of $f$ at any fixed point in $S$ has exactly one eigenvalue $1$ and one eigenvalue $-1$. Moreover, if the fixed locus of $f$ in $S$ is non-empty, then it is a smooth curve.
\end{lemma}
\proof
Since $f$ is an involution, it is known that $f$
 is linearizable near each fixed point, see the proof of Lemma 1.3 in \cite{Ka84}. Moreover, for the same reason, the linear form of $f$ is given by a diagonal matrix whose diagonal entries are equal to 1 or $-1$.  Since  $f^*\omega_S = -\omega_S$, there are exactly one diagonal entry equal to 1 and one diagonal entry equal to $-1$. The lemma follows easily. 
\endproof

Let ${\rm Inv}^- (S,P)$ denote the set of all involutions $f$ in ${\rm Aut}\, (S,P)$ such that $f^*\omega_S=-\omega_S$.
We have seen that  ${\rm Aut}\, (S,P)$ is contained in the three groups ${\rm Aut}\, (S_1,E_P)$, ${\rm Aut}\, (S_1,P')$ and ${\rm Aut}\, (S_1,P'')$  (Definitions \ref{def22} and \ref{def23}).   Observe also that  since $E_P \simeq \P^1$, there is a unique involution of $E_P$, not equal to the identity, which fixes the points $P'$ and $P''$. Denote this involution by $\iota_{E_P}$.  We deduce from the uniqueness of $\iota_{E_P}$ and the first assertion in Lemma \ref{l:inv} that the restriction of $f$ to $E_P$ coincides with $\iota_{E_P}$ for every $f$ in ${\rm Inv}^- (S,P)$.

Recall that
$$D_2 := C + 2C_{11} + E_2 + 2C_{12} + E_3 + 2C_{13} + 3F_1\,\, $$
is a singular fiber of $\varphi_{D_2}$.
Consider the divisor
$$D_2' := F_2 + 2C_{24} + F_3 + 2C_{34} + F_4 + 2C_{44} + 3E_4\, .$$ 
Using Figure 1, it is easy to check that $(D_2'^2)=0$ and 
 $(D_2'.R) = 0$ for each irreducible component $R$ of $D_2$ or $D_2'$. 
 We deduce that $D_2'$ is nef and we have
 $$(D_2^2)=(D_2'^2)=(D_2.D_2')=0\, .$$ 

 By Hodge index theorem, $(D_2')$ is proportional to $(D_2)$. Since $(D_2.C_{31})=(D_2'.C_{31})=1$, we obtain that  $|D_2|=|D_2'|$ and hence $D_2'$ is also a singular fiber of $\varphi_{D_2}$. 
According to Table 1 and Table A in \cite{Og89}, the
other singular fibers of $\varphi_{D_2}$ are exactly $a$ rational curves with an ordinary node and $b$ rational curves with an ordinary cusp, where $a$ and $b$ are some integers with $a + 2b = 8$. 

Let $\iota \in {\rm Aut}\, (S)$ be the inversion of the elliptic fibration $\varphi_{D_2} : S \to \P^1$ with respect to the section $C_{31}$.  The fibers of $\varphi_{D_2}$ are invariant by this map. Let $\Sigma$ be the fixed locus of $\iota$.
We will need the following lemmas.

\begin{lemma}\label{l:Sigma} 
The map $\iota$ belongs to ${\rm Inv}^- (S,P)$ and  its restriction  to $E_P$ is equal to $\iota_{E_P}$. Moreover, $\Sigma$ is the disjoint union of the smooth rational curves $C_{31}, C_{11}, C_{34}$ together with a smooth curve $\Sigma_0$ which contains no component of any  fiber of  $\varphi_{D_2}$.
\end{lemma}
\begin{proof}
By definition, $\iota$ is an involution of $S$ and $C_{31}$ is a component of $\Sigma$. Let $F_t$ denote the fiber of $\varphi_{D_2}$ over a point $t\in\P^1$. It meets $C_{31}$ at a unique point counting multiplicity because $C_{31}$ is a section of $\varphi_{D_2}$.
For any smooth fiber $F_t$, the differential of $\iota$ at the point $F_t\cap C_{31}$ 
is 1 in the direction of $C_{31}$ and $-1$ in the direction of $F_t$. It follows that 
$f^*\omega_S=-\omega_S$ and hence, 
by Lemma \ref{l:inv}, the components of $\Sigma$ are smooth and disjoint. 
In particular, if  a curve meets $C_{31}$, it cannot be a component of $\Sigma$ unless this is the curve $C_{31}$ itself. 

Since $D_2$ is invariant by $\iota$,  i.e., $\iota(D_2) = D_2$, and $C$ is its unique component meeting $C_{31}$,  it follows that $C$ is also invariant by $\iota$, i.e., $\iota(C)= C$.
Moreover, it follows from the above discussion that $\iota$ is not the identity on $C$. Since $P$ is the unique singular point of  $D_2$ in $C$, it is fixed by $\iota$. We conclude that $\iota$ belongs to ${\rm Inv}^- (S,P)$ and therefore its restriction  to $E_P$ is equal to $\iota_{E_P}$.
Observe also that any fiber $F_t$, different from $D_2$ and $D_2'$, is irreducible and meets the section $C_{31}$. So it cannot be a component of $\Sigma$. In order to complete the proof of the lemma, it is enough to determine the intersections of $\Sigma$ with $D_2$ and $D_2'$. 

Since the restriction of $\iota$ to $C$ is an involution which is not the identity, the derivative of $\iota|_C$ is $-1$ at each fixed point. In particular,  the derivative of $\iota|_C$ at $P$ is $-1$. Since $D_2$ is invariant by $\iota$, we deduce that $C_{11}$ is invariant by $\iota$ and therefore by Lemma \ref{l:inv}, the derivative of $\iota|_{C_{11}}$ at $P$ is 1. It follows that $\iota|_{C_{11}}$, which is an involution of a rational curve, should be identity. In other words, $C_{11}$ is a component of $\Sigma$. 

Using again the singularity $F_1\cap C_{11}$ of $D_2$ and derivatives of $\iota$ at this point, we obtain that  the restriction of $\iota$ to $F_1$ is not the identity and has exactly two fixed points including $F_1\cap C_{11}$. Observe now that there are exactly three singular points of $D_2$ on $F_1$; they are the intersections of $F_1$ with $C_{11}, C_{12}$ and $C_{13}$. This set of three points should be invariant by $\iota$. The only possibility here is that  $\iota$ permutes $F_1\cap C_{12}$ and $F_1\cap C_{13}$. Thus, it also permutes the pair $C_{12}, C_{13}$ and the pair $E_2,E_3$. 

We conclude that $C_{11}$ is the unique component of $D_2$ which is contained in $\Sigma$. Similarly, $C_{34}$ is the unique component of $D_2'$ which  belongs to $\Sigma$. To complete the proof of the lemma, it is enough to define $\Sigma_0$ as the complement in $\Sigma$ of $C_{31}\cup C_{11}\cup C_{34}$. 
\end{proof}

\begin{lemma}\label{lem40} 
The curve $\Sigma_0$ is irreducible and its genus is equal to $4$.  
\end{lemma}
\begin{proof}
By definition of $\iota$, if $F_t$ is a smooth fiber of $\varphi_{D_2}$, then $F_t\cap \Sigma$ is exactly the set of 2-torsion points. Since one of these points is $C_{31}\cap F_t$, the intersection $F_t\cap\Sigma_0$ is exactly the set of three other points.
Therefore,  the map
$$\varphi := \varphi_{D_2}|_{\Sigma_0} : \Sigma_0 \to \P^1\,\, $$
is of degree $3$ and unramified over $t$ when $F_t$ is a smooth fiber. We also deduce that $\Sigma_0\cap F_t$ contains three points counting multiplicity for each singular fiber $F_t$.

The description of the action of $\iota$ on $D_2$ in the proof of Lemma \ref{l:Sigma} shows that $\Sigma_0\cap D_2$ contains only one point which belongs to $F_1$. 
So this intersection is transversal since the multiplicity of $F_1$ in $D$ is equal to 3.
Similarly, $\Sigma_0\cap D_2'$ contains only one point of multiplicity 3. This point belongs to $E_4$ and the intersection $\Sigma_0\cap D_2'$  is transversal. The restriction of $\varphi$ to each component of $\Sigma_0$ is a non-constant map. 
Therefore, each component of $\Sigma_0$ meets $D_2$ and $D_2'$. It follows that $\Sigma_0$ is irreducible as it intersects $D_2$ and $D_2'$ transversally at only one point.

Consider a singular fiber $F_t$ which is a rational curve with an ordinary cusp. The restriction of $\iota$ to $F_t$ is an involution different from the identity. So it admits exactly two fixed points including $C_{31}\cap F_t$. In other words, $\Sigma_0\cap F_t$ contains exactly 1 point with multiplicity 3.

Consider now a singular fiber $F_t$ which is a rational curve with an ordinary node. Denote by $A$ the unique singular point of $F_t$. It should be fixed by $\iota$ and hence $A$ belongs to $\Sigma$. 
It follows that $A$ belongs to $\Sigma_0$ since $C_{31}$ intersects $F_t$ transversally at a unique point as it is a section of $\varphi_{D_2}$.
Moreover, since $F_t\cap \Sigma_0$ contains exactly 3 points counting multiplicity, we deduce that $F_t\cap \Sigma_0$ contains at most one point different from $A$.

The restriction of $\iota$ to $F_t$ is an involution different from the identity. Let $\pi: \P^1\to F_t$ be the normalization of $F_t$. The map $\iota|_{F_t}$ can be lifted by $\pi$ to an involution $\iota'$ of $\P^1$ which is not the identity.
So $\iota'$ admits exactly two fixed points that we denote by $a_1$ and $a_2$. The two points $\pi(a_1)$ and $\pi(a_2)$ belong to $F_t\cap \Sigma$ and  one of them should be $F_t\cap C_{31}$. 
We can assume that $\pi(a_1)=F_t\cap C_{31}$. 
We show that $\pi(a_2)\not=A$. Observe that the set $\pi^{-1}(A)$ is invariant by $\iota'$ and this set contains 2 points different from $a_1$. 
If  $\pi(a_2)=A$, then $a_2$ is one of those points and the other one should be fixed by $\iota'$ since $a_2$ is already fixed by $\iota'$. This contradicts the fact that $\iota'$ has only 2 fixed points. So $\pi(a_2)\not=A$ and we conclude that $F_t\cap \Sigma_0$ contains exactly two points : the point $A$ of multiplicity 2 and the point $\pi(a_2)$ of multiplicity 1.

It remains to compute the genus $g(\Sigma_0)$ of $\Sigma_0$. According to the above discussion, the map $\varphi$  is of degree $3$ and unramified over $t$ when $F_t$ is a smooth fiber. Moreover, $\varphi^{-1}(t)$ consists of $1$ point of multiplicity 3 if the fiber $F_t$ is $D_2$, $D_2'$ or a rational curve with an ordinary cusp. It consists of two points of multiplicities 1 and 2 if the fiber $F_t$ is a rational curve with an ordinary node. 
Recall that there are here  $b$ rational curves with an ordinary cusp and $a$ rational curves with an ordinary node.
Then, by Hurwitz's formula, we have
$$\chi(\Sigma_0) = 3\chi(\P^1) - 2-2 -2b -a$$
or equivalently
$$2 - 2g(\Sigma_0) = 6 -2 -2 -2b  -a\,\, .$$
Hence  $g(\Sigma_0) =4 $ because $a + 2b = 8$. This completes the proof.
\end{proof}

\begin{lemma}\label{lem41}
The centralizer of $\iota$  in ${\rm Aut}\, (S)$, denoted by  ${\rm C}_{\Aut(S)}(\iota)$, is a finite group.
\end{lemma}
\begin{proof} 
Let $f$ be an element of ${\rm C}_{\Aut(S)}(\iota)$. We have $f \circ \iota = \iota \circ f$. This identity implies that the fixed 
locus of $\iota$ is invariant by $f$.  Therefore, by Lemma \ref{lem40}, we have $f^*(\Sigma_0) = \Sigma_0$. It follows that ${\rm C}_{\Aut(S)}(\iota)$ preserves both the class $(\Sigma_0)$ in ${\rm NS}\, (S)$ and the orthogonal lattice $(\Sigma_0)_{{\rm NS}(S)}^{\perp}$  of $(\Sigma_0)$ 
in ${\rm NS}(S)$.

By the adjunction formula and using that $K_S=0$, we have 
$$(\Sigma_0^2) = (K_{\Sigma_0}) - (K_S.\Sigma_0) = {\rm deg}\, K_{\Sigma_0} = 2g(\Sigma_0) - 2=6\, .$$
It follows that $\Sigma_0$ is nef and big. 
By Hodge index theorem,  the lattice $(\Sigma_0)_{{\rm NS}(S)}^{\perp}$ is negative definite. In particular, the automorphism group ${\rm O}\,(\Sigma_0)_{{\rm NS}(S)}^{\perp})$ of the lattice $(\Sigma_0)_{{\rm NS}(S)}^{\perp}$ is finite. 
Therefore, the action of ${\rm C}_{\Aut(S)}(\iota)$
on $(\Sigma_0)_{{\rm NS}(S)}^{\perp}$ (and hence on ${\rm NS}\, (S)$) is finite. Finally, as in the proof of Lemma \ref{lem31}, using 
the global Torelli theorem and the finiteness of the canonical representation, we obtain that ${\rm C}_{\Aut(S)}(\iota)$ is finite.
\end{proof}

In what follows, we assume the following property.

\medskip\noindent
{\bf (S1)}  In Definition \ref{def23}, we choose $Q_1,\ldots, Q_k$ so that the set $\{Q_1,\ldots,Q_k\}$ is invariant by the involution $\iota_{E_P}$ of $E_P$. Necessarily, the number $k$ should be even  with $k \ge 2$.

\medskip

Since the maps in ${\rm Inv}^-(S,P)$ are equal to $\iota_{E_P}$ on $E_P$, they fix the set $\{Q_1,\ldots,Q_k\}$. Therefore, we can lift them to involutions of  the core surface $S_2$ (Definition \ref{def23}), or equivalently, 
we have 
$${\rm Inv}^-(S,P)\subset {\rm Aut}\, (S_2) \, .$$

Let $f_1$ be the automorphism of $S$  in Definition \ref{def32}. Set 
$$\iota_n := f_1^{-n} \circ \iota \circ f_1^{n}\quad \text{for} \quad n \in \Z\,\, .$$
Then $\iota_n$ is an involution in ${\rm Inv}^-(S,P)$ and therefore, it belongs to $\Aut(S_2)$.

\begin{lemma}\label{lem42}
Let $H$ be any subgroup of $\Aut(S_2)$ which contains $\iota_n$ for infinitely many $n \in \Z$ (the group  $\Aut(S_2)$ and the group $H$ in Lemma \ref{l:real-aut} below satisfy this property). Then these automorphisms $\iota_n$ define infinitely many conjugacy classes in $H$.
\end{lemma}
\begin{proof}
Assume that $\iota_n$ and $\iota_m$ are conjugate in $H$ for some $m, n \in \Z$. So there is $h\in H$ such that 
$$\iota_n = h^{-1} \circ \iota_m \circ h \, .$$ 
We have
$$f_1^{-n} \circ \iota \circ f_1^{n} = h^{-1} \circ f_1^{-m} \circ \iota \circ f_1^{m} \circ h\,\, .$$
It follows that  $f_1^{m} \circ h \circ f_1^{-n}$ belongs to ${\rm C}_{\Aut(S)}(\iota)$. So it belongs to the group 
$$\Gamma:={\rm C}_{\Aut(S)}(\iota)\cap \Aut(S,P) \, ,$$
 or equivalently, $h$ belongs to  the set 
$f_1^{-m} \Gamma f_1^{n}$.

Consider the representation 
$$r_{C, P} : {\rm Aut}\, (S, P) \to \C^{\times}$$ 
defined after Definition \ref{def21}. By Lemma \ref{lem41}, the group $\Gamma$ is finite. Therefore, its image by 
$r_{C, P}$ is a finite subset of $\C^\times$ that we will denote by $K_1$. 
Recall that in the coordinate $x$ on the curve $C$ that we used above, the map $f_1$ satisfies $f_1|_C(x) = 2x$. Hence
$r_{C, P}(h)$ belongs to the set  $2^{n-m}K_1$.

Finally, by Lemma \ref{l:S2-EP}, $r_{C, P}(h)$ belongs to a finite set that will be denoted by $K_2$. Therefore, $r_{C, P}(h)$ belongs to
$2^{n-m}K_1\cap K_2$.
As $K_1$ and $K_2$ are finite subsets of $\C^\times$, for each $n$, there are only finitely many $m$ such that $2^{n-m}K_1\cap K_2$ is non-empty. 
So for each $n$, there are only finitely many $m$ such that $\iota_m$ is conjugate to $\iota_n$ in $H$. The lemma then follows.  
\end{proof}

From now on, we assume further properties on our very special triple $(S,C,P)$ and the family of points $Q_1,\ldots, Q_k$  in {\bf (S1)}.

\medskip\noindent
{\bf (S2)} We specify the affine Weierstrass equation of $F$ with origin at $\infty$ as 
$$y'^2=x'(x'-1)(x'-\lambda) \quad \text{with} \quad (x',y')\in\C^2 \, ,$$
where $\lambda\in\R\setminus\{0,1\}$ such that $E$ and $F$ are not isogenous.

\medskip

Note that such a $\lambda$ exists because there are at most  countably many values of $\lambda$ which do not satisfy the last property. 
Note also that both $E$ and $F$  are defined over $\R$. Thus, the surface $S$ is defined over $\R$ as well.

In the coordinates $(x,y)$, the 2-torsion points of $E$ are $(0,0), (1,0), (2,0)$ and $(\infty,\infty)$. They are all real points. 
Similarly, in the coordinates $(x',y')$, the 2-torsion points of $F$ are $(0,0), (1,0), (\lambda,0)$ and $(\infty,\infty)$. They are also real points. 
It follows that all curves in Figure 1 are defined over $\R$ and their intersections are real points. 
In particular, the surface $S_1$ and the curve $E_P$ are defined over $\R$; the points $P'$ and $P''$ are real.

Since $\iota_{E_P}$ is the unique involution of $E_P$ which fixes $P'$ and $P''$, it is also defined over $\R$. Thus, the following choice is realizable.

\medskip\noindent
{\bf (S3)} We choose the set $\{Q_1,\ldots,Q_k\}$, invariant by $\iota_{E_P}$, in the real part of $E_P\setminus \{P',P''\}$. As a consequence, the core surface $S_2$ is defined over $\R$.

\medskip 

Observe that 
$${dx\over y }\wedge {dx'\over y' }$$
defines a holomorphic 2-form on $E\times F$. It induces a holomorphic 2-form $\omega_S'$ on $S$ which is proportional to $\omega_S$ because $H^0(S,\Omega_S^2)=\C\omega_S$.

\medskip\noindent
{\bf (S4)}  We take $\omega_S:=\omega_S'$. In particular, the form $\omega_S$ is  also defined over $\R$. 

\begin{lemma} \label{l:real-aut}
Let 
$$H := \big\{f \in \Aut(S_2)\, |\, f^*(\omega_S)=\pm \omega_S \big\}\,\, .$$
Then $H$ contains the set $\Inv^-(S,P)$ and is a finite-index subgroup of $\Aut(S_2)$. Moreover, all elements of $H$ are defined over $\R$. 
\end{lemma}
\proof
Recall that $\Aut(S_2)$ is identified to a subgroup of $\Aut(S)$.  It is clear that $H$ is a finite index subgroup of $\Aut(S_2)$ by the finiteness of the canonical representation of $\Aut(S)$. The fact that $H$ contains $\Inv^-(S,P)$ is  also clear by definition of the later set. Consider an arbitrary element $f$ of $H$. It remains to show that it is defined over $\R$.

Since $S$ is defined over $\R$, the complex conjugation on the complex number field induces an anti-holomorphic involution $c:S\to S$. We need to check that $c\circ f\circ c=f$. Observe that  $c\circ f\circ c$ is holomorphic. So it belongs to $\Aut(S)$.

By (S4), we have $c^*\omega_S=\overline\omega_S$ and $c^*\overline\omega_S=\omega_S$. Since $f$ is an element of $H$, we also have $f^*(\overline\omega_S)=\pm\overline\omega_S$ and hence  $(c\circ f\circ c)^*\omega_S=f^*\omega_S$.

Let $R$ be any curve among $24$ curves in Figure 1. Since $R$ is defined over $\R$, we have $c^*(R) = R$. Recall that these 24 curves  generate ${\rm NS}\, (S) \otimes \Q$ and ${\rm NS}\, (S)$ is torsion free. It follows that $c^*({\rm NS}\, (S)) \subset {\rm NS}\, (S)$ and moreover
$$c^*|_{{\rm NS}\, (S)} = \id_{{\rm NS}\, (S)}\,\, .$$ 
Hence $(c \circ f \circ c)^* = f^*$ on ${\rm NS}\, (S)$ as well. Now, as in the proof of  Lemma \ref{lem31}, by applying the global Torelli theorem, we get  $c\circ f\circ c=f$. This ends the proof of the lemma.
\endproof

\medskip\noindent
{\bf (S5)} In the proof of Theorem \ref{thm21}, we choose the variety $M$ of dimension $m$ so that it is defined over $\R$.  In particular, the variety $V:=S_2\times M$ is also defined over $\R$.

\medskip\noindent
{\it End of the proof of Theorem \ref{thm1}.} 
Define $V:=S_2$ for $d=2$, and $V:=S_2\times M$ for $d\geq 3$ and $m:=d-2$. 
First, observe that the Kodaira dimension of $S_2$ is 0 and 
the Kodaira dimension of $M$ is $d-2$. It follows that the Kodaira dimension of $S_2\times M$ is also equal to $d-2$, see \cite[p.69]{Ue75}. So the assertion about the Kodaira dimension of $V$ in Theorem \ref{thm1} holds. 
By Theorem \ref{thm21}, the properties (1) and (2) of Theorem \ref{thm1} are true as well.

Let $H$ be the group in Lemma \ref{l:real-aut}. Observe that the involution $\iota_n$ belongs to $H$ for every $n$. 
By Lemma \ref{lem42}, $H$ contains infinitely many conjugacy classes of involution.  
Define $G:=H$ for $d=2$ and $G:=H\times \{\id_M\}$ for $d\geq 3$. So $G$ contains infinitely many of conjugacy classes of involution.

Since $\Aut(M)$ is finite and $\Aut(V)=\Aut(S_2)\times \Aut(M)$, by Lemma \ref{l:real-aut}, $G$ is a finite-index subgroup of $\Aut(V)$. 
Also by the same lemma, all elements of $G$ are invariant under the action induced by the complex conjugation. Now, by \cite[Lemma 13]{Le17}, $V$ admits infinitely many real forms which are mutually non-isomorphic over $\R$. So the property (3) in Theorem \ref{thm1} is satisfied. This completes the proof of the theorem.
\hfill $\square$

\begin{remark} \label{r:field} \rm
Instead of the choice in (S2), we can choose an elliptic curve $F$ defined over $\Q$, non-isogenous to $E$, so that its 2-torsion points are all rational. For example, we can take the normalization of the singular curve $E_\eta'$ from \cite{COV15}.  This choice insures that the main objects in the proof of Theorem \ref{thm1}, namely, the core surface $S_2$, the involutions $\iota_n$  and the variety $M$, are (or can be chosen to be) defined over $\Q$. As a consequence, Theorem \ref{thm1}(1)-(2) holds for these $S_2$ and $S_2 \times M$ over any field of characteristic 0 and in Theorem \ref{thm1}(3), we can replace $\R$ by any field $K$ of characteristic 0 and $\C$ by any quadratic extension of $K$, as in \cite{Le17}.
\end{remark}




\end{document}